\documentclass[11pt]{article}
\usepackage{cite}
\usepackage{hyperref}
\usepackage{appendix}
\usepackage{thmtools}  
\usepackage{amssymb,amsmath,amsfonts,mathrsfs,amsthm,epsfig,latexsym,color}
\usepackage{bm}
\usepackage{enumitem,geometry}
\usepackage{fourier}
\usepackage[all]{xy}

\makeatletter
\newcommand{\subsectionruninhead}{\@startsection{subsection}{2}{0mm}
	{-\baselineskip}{-0mm}{\bf\large}}
\newcommand{\subsubsectionruninhead}{\@startsection{subsubsection}{3}{0mm}
	{-\baselineskip}{-0mm}{\bf\normalsize}}
\makeatother

\newtheorem*{theorem*}{Theorem}
\newtheorem*{proof*}{Proof}
\newtheorem*{proposition*}{Main Proposition}
\newtheorem*{question*}{Question}
\newtheorem*{notation*}{Notation}
\newtheorem*{corollary*}{Corollary}
\newtheorem*{claim*}{Claim}
\newtheorem*{remark*}{Remark}

\newtheorem{conjecture}{Conjecture}

\newtheorem{proposition}{Proposition}[section]
\newtheorem{theorem}[proposition]{Theorem}
\newtheorem{corollary}[proposition]{Corollary}
\newtheorem{lemma}[proposition]{Lemma}
\newtheorem{claim}[proposition]{Claim}

\theoremstyle{definition}

\theoremstyle{remark}
\newtheorem{remark}[proposition]{Remark}

\numberwithin{equation}{section}

\def\NN{\mathbb{N}}
\def\RR{\mathbb{R}}

\def\TT{\mathbb{T}}
\def\ZZ{\mathbb{Z}}

\def\tildeF{\tilde{\mathcal{F}}}
\def\e{{\varepsilon}}

\def\tildeF{\widetilde{\mathcal{F}}}

\begin{document}

	\title{Smooth Stable Foliations of Anosov Diffeomorphisms}
	\author{Ruihao Gu\footnote{Email: ruihaogu@fudan.edu.cn}}
	\date{\today}
	\maketitle
	
	\begin{abstract}
		In this paper, we focus on the rigidity of $C^{2+}$-smooth codimension-one stable foliations of Anosov diffeomorphisms. Specifically, we show that if the regularity of these foliations is slightly bigger than $2$, then they will have the same smoothness as the diffeomorphisms.
	\end{abstract}

	\section{Introduction}
	
	Let $f$ be a $C^r$-smooth ($r\geq 1$) Anosov diffeomorphism of a smooth closed Riemannian manifold $M$, i.e, there exists a $Df$-invariant splitting $TM=E^s_f\oplus E^u_f$ such that $Df$ is contracting on $E_f^s$ and expanding on $E_f^u$, uniformly. It is well known that the distributions $E_f^s$ and $E_f^u$ are H\"older continuous and uniquely integrable to foliations $\mathcal{F}^s_f$ and $\mathcal{F}^u_f$, respectively  with  $C^r$-smooth leaves varing continuously with respect to $C^r$-topology. However, the regularity of these foliations may not be $C^r$. Indeed, if the regularity  $r\geq 2$ of $f$,  the foliation $\mathcal{F}^s_f$ (or symmetrically $\mathcal{F}^u_f$) is absolutely continuous and if we further assume that it is codimension-one, then it is $C^1$-smooth \cite{AnosovC2,hirschpugh,Pesin04}  but could not be $C^2$ in general \cite{AnosovC2}. 
	
	In 1991, Flaminio and Katok \cite{FK91} proposed the following conjecture.
	
	\begin{conjecture}[\cite{FK91}]
		If the foliations $\mathcal{F}^s_f$ and $\mathcal{F}^u_f$ of a $C^k$-smooth Anosov diffeomorphism $f:M\to M$ are both $C^2$-smooth, then $f$ is $C^{\max\{2,k\}}$ conjugate to a hyperbolic automophism of an infra-nilmainfold. 
	\end{conjecture}
	
	This conjecture can be divided into two parts which are still open. One is the famous conjecture of  Smale \cite{Smale67} which expects to classify the Anosov diffeomorphisms in the sense of topology, i.e, every Anosov diffeomorphism is topologically conjugate to a hyperbolic automophism of an infra-nilmainfold. The other one is a rigidity issue, i.e, whether the smooth foliation leads to a higher regularity of the conjugacy or not, since the conjugacy between two Anosov diffeomorphisms is usually H\"older continuous only.
	
	The topological classification conjecture has some evidences \cite{Franks1970,Manning1974,Newhouse1970}. For instance, when $E^s_f$ (or $E^u_f$) is codimension-one, then $M$ is a torus. Moreover,  under the assumption of $M$ is a nilmanifold, $f$ is conjugate to a hyperbolic algebraic model. These are also why  researches of rigidity usually focus on the toral Anosov diffeomorphisms. The rigidity issue has been extensively and deeply studied under some restriction of Lyapunov exponent \cite{dela,GG2008,yangradu}. However we know few of the rigidity on smooth foliation. Indeed, as far as authors know, it has only partial answer in \cite{FK91,katokhurder,Ghys93,dela}. 
	
	In \cite{FK91}, Flaminio and Katok proved that  a volume-preserving Anosov diffeomorphism $f$ of  $2$-torus $\TT^2$  with $C^r (r\geq 2)$ stable and unstable foliations is $C^r$ conjugate to a linear one. Moreover, they obtained a similar result for an Anosov diffeomorphism $f$ of  $\TT^4$ preserving a symplectic form with  $C^{\infty}$-smooth stable and unsable foliations.  However de la Llave \cite[Theorem 6.3]{dela} constructed a counterexample on  $\TT^d (d\geq 4)$, precisely, for any $k\in\NN$ there exist hyperbolic automorphism $A:\TT^d\to\TT^d$, Anosov diffeomorphism $f:\TT^d\to\TT^d$ and a $C^k$-conjugacy $h$ between $f$ with $A$ such that $h$ is not $C^{k+1}$-smooth.
	
	As a corollary of \cite{FK91},  $C^2$-regularity of hyperbolic foliation on $\TT^2$ implies higher-regularity of itself. In a same sense of such bootstrap of foliation, Katok and Hurder \cite{katokhurder} proved that for a  $C^r(r\geqslant5)$ volume-preserving Anosov diffeomorphism $f$ of $\TT^2$, if distributions $E_f^{s\setminus u}$  are $C^{1,\omega}$, i.e, the derivatives are are respectively of class $\omega(s)=o(s|{\rm log}(s)|)$,   then $\mathcal{F}^{s/u}_f$ are actually $C^{r-3}$-smooth and $f$ is $C^{r-3}$-conjugate to a toral hyperbolic automorphism. Similarly, Ghys \cite{Ghys93} showed that  for a $C^r \ (r\geq 2)$ Anosov diffeomorphism $f$ of $\TT^2$, if the stable foliation $\mathcal{F}^s_f$ is $C^{1+{\rm Lip}}$-smooth, then it is actually $C^r$-smooth.

	Our aim in this paper is getting higher regularity of codimension-one hyperbolic foliations  under the assumption of more or less $C^2$-smoothness, see Theorem \ref{thm2} and Theorem \ref{thm1}. In particular, we get some rigidity results on $\TT^2$. Let us give two notations.  We denote by $\lambda_f^u(x)$ the sum of Lyapunov exponents (if it exists) of $f$ on the unstable subbundle at the point $x$, namely,
	\begin{align*}
		\lambda^u_f(x)=\lim_{n\to +\infty} \frac{1}{n}\mathrm{log}\Big|{\rm det}\big(Df^n \arrowvert_{E^u_f(x)}\big)\Big|.
	\end{align*}
	For $r>1$, let  $r_*=\Big\{ \begin{array}{lr}
		r-1+{\rm Lip}, \; \  \ \qquad r\in\NN  \\ \;\;\; r,\ \ \ \ \ \ \  r\notin  \NN \ {\rm or}\ r=+\infty
	\end{array}$.

	\begin{theorem}\label{thm2}
		Let  $f:\TT^d\to\TT^d\ (d\geq 2)$ be a $C^r\ (r>2)$ Anosov diffeomorphism with the $(d-1)$-dimension $C^{2+\e}$-smooth ($\e>0$) stable foliation $\mathcal{F}_f^s$. Then $\mathcal{F}_f^s$ is $C^{r_*}$-smooth  and $\lambda^u_f(p)\equiv \lambda^u_A$, for all periodic point $p$ of $f$, where $A$ is the linearization of $f$.
	\end{theorem}
	
	\begin{remark}
		Here we briefly explain why we just get $C^{r_*}$-smoothness. In this paper, the regularity of foliation is given by foliation chart, see Section 2 for precise definition.  Instead of the regularity of local chart, we will first prove that the foliation has $C^r$-smooth holonomy.  However, the regularity of foliation may be lower than its holonomy, e.g., see \cite[Section 6]{PSW}.
	\end{remark}

	In particular, we have the following corollary linking the regularity of foliation with Lyapunov exponents of its transversal.
	\begin{corollary}\label{cor1}
		Let $f$ be a $C^r(r>2)$ Anosov diffeomorphism of \,$\mathbb{T}^d\ (d\geq2)$ with the $(d-1)$-dimension stable foliation $\mathcal{F}_f^s$ and linearization $A:\TT^d\to \TT^d$. Then the followings are equivalent:
		\begin{enumerate}
			\item There exists small $\e>0$ such that $\mathcal{F}_f^s$ is $C^{2+\e}$-smooth;
			\item For all periodic point $p$ of $f$, $\lambda^u_f(p)\equiv \lambda^u_A$;
			\item The foliation $\mathcal{F}_f^s$ is $C^{r_*}$-smooth.
		\end{enumerate}
	\end{corollary}
	

	\begin{remark}
		By the same way of proving `` $2\Longrightarrow 3$" in Corollary \ref{cor1}, one can get an interesting result for non-invertible Anosov maps, that is their  codimension-one unstable foliations are always smooth as  nearly same regularity as  maps. Concretely, for a $C^r\ (r>1)$ non-invertible Anosov endomorphism $f:\TT^d\to\TT^d$   with one-dimensional stable bundle, if there exists unstable foliation $\mathcal{F}_f^u$ of $f$, then $\mathcal{F}_f^u$ is  $C^{r_*}$-smooth. Indeed, by \cite{AGGS23}, the existence of  $\mathcal{F}_f^u$ implies $\lambda_f^s(p)=\lambda_A^s$ for all $p\in{\rm Per}(f)$. Then the proof of Corollary \ref{cor1} leads to $C^{r_*}$ regularity of $\mathcal{F}_f^u$.
	\end{remark}

	We mention in advance that our method to prove Theorem \ref{thm2} is different from one of \cite{FK91,katokhurder,Ghys93}. Indeed, we will
	consider a diffeomorphism of circle $S^1$   induced by codimension-one foliation and apply KAM theory (see Theorem \ref{thm5}) to it. Hence the regularity $C^{2+\e}$ of foliation is in fact a condition of the induced circle diffeomorphism for using KAM. Particularly,  when $\TT^d$  is restricted to be $\TT^2$, we can lower the regularity of our assumption to be $C^{1+{\rm AC}}$, i.e, the derivative of foliation charts are absolutely continuous. 
	
	\begin{theorem}\label{thm1}
		Let $f:\TT^2\to\TT^2$ be a $C^r(r\geqslant2)$ Anosov diffeomorphism  with the $C^{1+{\rm AC}}$-smooth stable foliation $\mathcal{F}_f^s$. Then $\mathcal{F}_f^s$ is $C^{r_*}$-smooth and $\lambda^u_f(p)\equiv \lambda^u_A$, for all periodic point $p$ of $f$, where  $A$ is the linearization of $f$.	
	\end{theorem}

	By combining our result and a rigidity result of R. de la Llave \cite{dela} which says that constant periodic Lyapunov exponents  implies smooth conjugacy on $\mathbb{T}^2$, we have following two direct corollaries.
	
	\begin{corollary}
		Let $f$ be a $C^r(r\geqslant2)$ Anosov diffeomorphism of\, $\mathbb{T}^2$. If the stable and unstable foliations of $f$ are both $C^{1+{\rm AC}}$, then $f$ is $C^{r_*}$ conjugate to its linearization. In particular, $f$ preserves a smooth volume-measure.
	\end{corollary}
	
	\begin{corollary}
		Let $f$ be a $C^r(r\geqslant2)$ volume-preserving Anosov diffeomorphism of\, $\mathbb{T}^2$. If the stable foliation of $f$ is $C^{1+{\rm AC}}$, then $f$ is $C^{r_*}$ conjugate to its linearization.
	\end{corollary}
	
	\section{Preliminaries}
	
	As usual, a foliation $\mathcal{F}$ with dimension $l$ of a closed  Riemannian manifold $M=M^d$ is $C^r$-smooth, if there exists a set of $C^r$ local charts $\{(\phi_i, U_i)\}_{i=1}^k$ of $M$ such that $\phi_i:D^l\times D^{d-l}\to U_i$ satisfying \[\phi_i\big(D^l\times \{y\}\big) \subset \mathcal{F}\big(\phi_i(0,y)\big), \quad \forall y\in D^{d-l},\] where $D^{l}$ and  $D^{d-l}$ are open disks with dimension $l$ and  $d-l$ respectively and the chart $(\phi_i, U_i)$ is called a $C^r$ \textit{foliation chart}.
	
	Let $f$ be a $C^r$-smooth Anosov diffeomorphism  of $d$-torus $\TT^d$, i.e., there is a $Df$-invariant  splitting 
	$$T\mathbb{T}^d=E^s_f\oplus E^u_f,$$ and constants $C,\lambda>1$, such that for all $n>0$, $x\in\TT^d$ and $v^{s/u}\in E^{s/u}_f$
	\[\| D_xf^n(v^s)  \| \le C\lambda^{-n}\|v^s\|, \quad {\rm and} \quad  \| D_xf^n(v^u)  \| \geq C\lambda^{n}\|v^u\|.\]
	Note that $f_*:\pi_1(\TT^d)\to \pi_1(\TT^d)$ also induces a  hyperbolic automorphism $A:\TT^d\to \TT^d$\cite{Franks1970} which is called the \textit{linearization} of $f$. Denote  the $A$-invariant hyperbolic splitting by \[T\mathbb{T}^d=E^s_A\oplus E^u_A.\]
   Let $i\in \{f,A\}$.  Denote the foliations tangent to $E_i^s$ and  $E_i^u$ by $\mathcal{F}_i^s$ and $\mathcal{F}_i^u$ respectively. It is known that the leaf $\mathcal{F}_f^\sigma(x) \ (\sigma=s,u)$ is an immersed $C^r$ submanifold of $\mathbb{T}^d$. 	 
   
   Since $f$ and $A$ are always conjugate \cite{Franks1970}, we denote the conjugacy by  $h:\mathbb{T}^d\to\TT^d$, namely, $h$ is a homeomorphism such that
   \begin{align*}
   	h\circ f=A\circ h
   \end{align*}
   By the topological character of (un)stable foliation, i.e.,
   \begin{align*}
   	\mathcal{F}_f^s(x)=\big\{y\in\TT^d:d\big(f^{n}(x),f^{n}(y)\big)\to 0 , n\to +\infty \big\},
   \end{align*}
   $h$ preserves the foliations, that is for all $x\in\TT^d$,
   \begin{align*}
   	h\big(\mathcal{F}_f^u(x)\big)=\mathcal{F}_A^u\big(h(x)\big) ,\quad h\big(\mathcal{F}_f^s(x)\big)=\mathcal{F}_A^s\big(h(x)\big) . 
   \end{align*}

   It is convenient to observe the foliations on the universal cover $\RR^d$. Let $\pi:\RR^d\to \TT^d$ be the natural projection. Denote by $F, A$ and $H:\RR^d\to \RR^d$ the lifts of $f,A$ and $h:\TT^d\to \TT^d$ respectively.  For convenience, we can assume that $H(0)=0$. We denote the lift of $\mathcal{F}^{\sigma}_{i} \ (\sigma=s/u, i=f/A)$ on $\RR^d$ by $\tildeF^{\sigma}_{i}$ which are also the stable/unstable foliation of the lift $F/A$.  Recall that $H\big(\tildeF^{\sigma}_f\big)=\tildeF^{\sigma}_A$, $\sigma=s/u$ and hence $\tildeF^{s}_f$ and $\tildeF^{u}_f$ admit the  Global Product Structure just like $\tildeF^{s}_A$ and $\tildeF^{u}_A$, i.e., each pair of leaves $\tildeF^{s}_f(x)$ and $\tildeF^{u}_f(y)$ transversally intersects at exactly one point. Then we can define the holonomy map Hol$^s_i\ (i=f/A)$ induced by the foliation $\tildeF^{s}_i$ as 
   \[ {\rm Hol}^s_{i,x,y}:\tildeF^u_i(x)\to \tildeF^u_i(y), \quad {\rm Hol}^s_{i,x,y}(z)=\tildeF^s_i(z)\cap \tildeF^u_i(y). \]
   Note that the holonomy map Hol$^s_f$ and foliation $\tildeF^s_f$ are both absolutely continuous\cite{Pesin04}. 
   As mentioned before, the regularity of foliation may be lower than one of its holonomy. However, we still have the following lemma. We refer to  \cite[Section 6 ]{PSW} for more details about the next lemma and also the counterexample of foliations whose regularity is strictly lower than the holonomy. 
   
   \begin{lemma}[\cite{PSW}]\label{lem holonomy to foliation}
   Let $f:\TT^d\to \TT^d$ be a $C^r$-smooth $(r\geq 1)$ Anosov diffeomorphisms. Then
   \begin{enumerate}
   	\item If the holonomy maps ${\rm Hol}^s_f$ of $\tildeF^s_f$ are uniformly $C^r$-smooth, i.e., for any $x,y,z\in\RR^d$ the holonomy map ${\rm Hol}^s_{f,x,y}$ is $C^r$ and its derivatives (with respect to $z$) of order $\leq r$  vary continuously  with respect to $(x,y,z)$. Then the foliation $\mathcal{F}^s_f$ are $C^{r_*}$-smooth. 
   	\item If $\mathcal{F}^s_f$ is a $C^k$-smooth  ($k\leq r$) foliation, then the holonomy maps ${\rm Hol}^s_f$ of $\tildeF^s_f$ are uniformly $C^k$-smooth.
   \end{enumerate} 
   \end{lemma}
   
   \begin{remark}
   Note that the second item of Lemma \ref{lem holonomy to foliation}  is trivial. And the first item  is just an application of Journ$\mathrm{\acute{e}}$'s  lemma \cite{regularity} which asserts that the regularity of a diffeomorphism can be obtained from the uniformly regularity of its restriction on two tranverse foliations with uniformly smooth leaves. Indeed, considering a point $x\in\RR^d$ and let $\alpha:D^l \to \tildeF^s_f(x)$ and $\beta:D^{d-l} \to \tildeF^u_f(x)$ be two $C^r$-parameterization such that $x= \big(\alpha(0),\beta(0) \big)$. Then \[\phi(a,b):={\rm Hol}^s_{f,\alpha(0),\alpha(a)}\big(\beta(b)\big)\] gives us a foliation chart whose derivatives along $D^l$ and $D^{d-l}$ are both $C^r$. Hence by Journ$\mathrm{\acute{e}}$'s  lemma, it is a $C^{r_*}$-foliation chart.
   \end{remark}

   On the other hand, the regularity of holonomy induced by $\tildeF^s_f$ can be given by the smoothness of the conjugacy $H$ restricted on the transversal direction.
   
   \begin{lemma}\label{lem conjugacy to holonomy}
   	Assume that the conjugacy $H:\RR^d\to \RR^d$ is uniformly $C^r$-smooth along the unstable foliation $\tildeF^u_f$, then the holonomy ${\rm Hol}^s_f$ is uniformly $C^r$-smooth.
   \end{lemma}
  \begin{proof}
	For given $x,y\in\RR^d$ and $z\in \tildeF_f^u(x)$, the holonomy map satisfies
	\[ {\rm Hol}^s_{f,x,y}(z)= H^{-1}\circ {\rm Hol}^s_{A,H(x),H(y)}\circ H(z), \]
	since $H$ preserves the foliations. Note that the holonomies ${\rm Hol}^s_A$ induced by $\tildeF^s_A$ are actually translations. Therefore the holonomies ${\rm Hol}^s_{f}$ has the same regularity as $H|_{\tildeF^u}$.
  \end{proof}
  
  Combining Lemma \ref{lem holonomy to foliation} and Lemma \ref{lem conjugacy to holonomy}, we can get Theorem \ref{thm2} and Theorem \ref{thm1} by proving that $H$ is $C^r$-smooth along the unstable leaves.  Precisely, we will  prove the following property.
  
  \begin{proposition}\label{prop H smooth}
  		Let  $f:\TT^d\to\TT^d$be a $C^r$-smooth Anosov diffeomorphism with the $(d-1)$-dimension $C^k$-smooth $(k<r)$ stable foliation $\mathcal{F}_f^s$. If one of the followings holds,
  		\begin{enumerate}
  			\item $r> 2$  and  $k>2$; 
  			\item $d=2, r\geq 2$ and $k=1+{\rm AC}$.
  		\end{enumerate}
  		Then the conjugacy $H$ between lifts $F$ and $A$ is uniformly $C^r$-smooth along each unstable leaf. 
  \end{proposition}

  We will prove this proposition in Section 3. Before that, we note that to get $C^r$-regularity of $H|_{\tildeF^u_f}$, we can just prove a lower one. Indeed, by an enlightening work of de la Llave\cite{dela}, one can get the $C^r$-smoothness from the absolute continuity, see  \cite[Lemma4.1, Lemma 4.5 and Lemma 4.6]{dela}. Here we state it for convenience.
    
   \begin{lemma}[\cite{dela}]\label{lem AC to Cr}
   	Let $f:\TT^d\to\TT^d$be a $C^r$-smooth Anosov diffeomorphism with one-dimensional unstable foliation $\mathcal{F}^u_f$. If one of the followings holds,
   	\begin{enumerate}
   		\item  The conjugaies $H$ and $H^{-1}$ are absolutely continuous restricted on the unstable foliation $\tildeF^u_f$ and $\tildeF^u_A$, respectively.
   		\item For all periodic point $p$ of $f$, $\lambda^u_f(p)\equiv \lambda^u_A$;
   	\end{enumerate}
   	  Then the restriction $H|_{\tildeF^u_f}$ is uniformly $C^r$-smooth.
   \end{lemma}

	Now we can finish the proof of our main theorems.
	
	\begin{proof}[Proof of Theorem \ref{thm2}, Corollary \ref{cor1} and Theorem \ref{thm1}]
	Let $f$ satisfy the condition of Theorem  \ref{thm2} or Theorem \ref{thm1}.	 By Proposition \ref{prop H smooth}, $H|_{\tildeF^u_f}$ is smooth, so is $h|_{\mathcal{F}^u_f}$. 
	 Hence  $\lambda^u_f(p)\equiv \lambda^u_A$, for all periodic point $p$ of $f$. Combining with Lemma \ref{lem holonomy to foliation} and Lemma \ref{lem conjugacy to holonomy}, we can get these two theorems immediately. 
	 
	 Note that  ``$(1)\implies (2)$"  of Corollary \ref{cor1}  is given by Theorem \ref{thm1},  ``$(2)\implies (3)$"  is guaranteed by the case 2 of Lemma \ref{lem AC to Cr} and  ``$(3)\implies (1)$"  is trivial.
	\end{proof}

	We will get the absolute continuity of $H$ restricted on unstable leaves by applying following KAM theory. 	Let $R_{\alpha}:\mathbb{R}\to\RR$ be the translation on $\mathbb{R}$ such that $R_{\alpha}(x)=x+\alpha,x\in \mathbb{R}$.  Denote the induced  rigid rotation on $S_1$ by $\overline{R_{\alpha}}$. 

	\begin{theorem}[\cite{KO,KT}] \label{thm5}
		Let $T$ be an  orientation-preserving circle diffeomorphism with  irrational  rotation number $\alpha$ which is algebraic. Then one has the following two properties
		\begin{enumerate}
			\item If  $T$ is $C^{2+\epsilon}$-smooth. Then $T$ is $C^{1+\epsilon-\delta}$-smoothly conjugate to $\overline{R_{\alpha}}$ for any $\delta>0$. 
			\item If  the pair $(T,\alpha)$ satisfies the K.O. condition \cite{KO}, in particular, the conditions: $T$ is $C^{1+{\rm AC}}$, $T''/T'\in L^p$ for some $p>1$  and  deg$(\alpha)=2$  satisfy the K.O. condition.
			Then $T$ is absolutely continuously conjugate to $\overline{R_{\alpha}}$.    
		\end{enumerate}
	\end{theorem}
	  In the case 2, the fact \cite{diophan} that deg$(\alpha)=2$ if and only if $\alpha$ has a periodic simple continued fraction expansion is helpful to check the K.O. condition.

	\section{Absolutely continuous rotation induced by smooth foliation}             
	
	In this section, we will obtain our main result Proposition \ref{prop H smooth}.  To prove it, as mentioned before, we can consider the circle  diffeomorphism induced by the codimenison-one foliation $\tildeF^s_f$ and show it is smooth conjugate to a rigid rotation given by $\tildeF^s_A$. We use the same notations as Section 2. 
	
	For reducing the action of $\tildeF^s_f$ on $\tildeF^u_f(0)$ to action on $S^1$, one can apply the $\ZZ^d$-actions. By the Global Product Structure,  the following map $T_n^i$ is well defined. 
	For $n\in \mathbb{Z}^d$ and $i\in\{f,A\}$,
	\begin{align*}
		T_i^n:\tildeF_i^u(0) \to \tildeF_i^u(0),\qquad \qquad \quad  \\
		T_i^n(x)=\tildeF_i^u(0)\cap \tildeF_i^s(x+n),\quad \forall x\in\tildeF_i^u(0). 
	\end{align*}
	Note that  for each $n\in\mathbb{Z}^d$, $T_i^n(x)={\rm Hol}^s_{i,n,0}\circ R_n(x)$ where $R_n(x)=x+n$.

	\begin{proposition}\label{prop}
		Assume that $\mathcal{F}^s_f$ is a $C^k$-smooth foliation. Then  for each $n\in\mathbb{Z}^d$, $T_i^n$ is $C^k$-smooth, $i\in\{f,A\}$. Moreover, one has:
		\begin{enumerate}
			\item $T_i^n\circ T_i^m=T_i^{n+m}$ for all $n,m\in \mathbb{Z}^d$;
			\item $H\circ T_f^n=T_A^n\circ H$ for all $n\in \mathbb{Z}^d$.
		\end{enumerate}
	\end{proposition}

	\begin{proof}
		The regularity of $T_i^n={\rm Hol}^s_{i,n,0}\circ R_n(x)$ is directly from the $C^k$-holonomy,  since the holonomy is smoother than the foliation (see Lemma \ref{lem holonomy to foliation}). And $\{T_i^n\}_{n\in\mathbb{Z}^d}$ is commutative by the fact that the holonomy maps  are commutative with the $\mathbb{Z}^d$-actions on $\mathbb{R}^d$, i.e., \[R_m\circ {\rm Hol}^s_{i,x,y}(z)={\rm Hol}^s_{i,R_m(x),R_m(y)}\circ R_m, \quad  \forall m\in\mathbb{Z}^d, \forall x,y\in\RR^d \ {\rm and}\  \forall z\in \tildeF^u_i(x),\]
		since $\tildeF^{s/u}_i(x+m)=\tildeF^{s/u}_i(x)+m$ for all $x\in\RR^d$ and $m\in\ZZ^d$.  Hence 
		\begin{align*}
			T_i^n\circ T^i_m &= {\rm Hol}^s_{i,n,0}\circ R_n\circ {\rm Hol}^s_{i,m,0}\circ R_m \\ &= {\rm Hol}^s_{i,n,0}\circ  {\rm Hol}^s_{i,n+m,n}\circ R_{n+m}
			={\rm Hol}^s_{i,n+m,0}\circ R_{n+m}=T^{n+m}_i.
		\end{align*}
	  Recall that we can assume  $H(0)=0$. 	Note that $H$ preserves the foliation and satisfies $H\circ R_m=R_m\circ H$,  for all $m\in\mathbb{Z}^d$.  Hence
		\[	H\circ T^n_f=H\circ  {\rm Hol}^s_{f,n,0}\circ R_n=  {\rm Hol}^s_{A,n,0}\circ H\circ  R_n=  {\rm Hol}^s_{A,n,0}\circ R_n\circ  H= T^n_A\circ H.\]
      This  completes the proof of proposition.
	\end{proof}

	Now we are going to prove Proposition \ref{prop H smooth}. 	Let $\{ e_i\}_{i=1}^d$ be an orthonormal basis of $\RR^d$. We will reduce a pair of conjugate $\mathbb{Z}^d$-actions, for instance $(T^{e_1}_f, T^{e_1}_A)$, to be a pair of conjugate circle diffeomorphisms and show the conjugacy is absolutely continuous by applying the KAM theory (Theorem \ref{thm5}). 
	This method has a similar spirit with one used by Rodriguez Hertz, F. in [\citen{FRH}].

	\begin{proof}  [Proof of Proposition \ref{prop H smooth}]
		
		We pick two unit vector of the normal orthogonal basis, for example $e_1, e_d$. Assume that $\mathcal{F}^s_f$ is a $C^k$-smooth codimension-one foliation, where $k$ satisfies the condition of proposition. Firstly, we use $T^{e_1}_f$ to construct a $C^k$ circle diffeomorphism.
		We still denote the translation on $\RR$ by $R_\alpha(x)=x+\alpha$ and the natural projection by $\pi:\RR\to S^1$ for short. 
		
		\begin{claim}\label{claim f circle map}
			There exists a $C^k$ diffeomorphism $h_f:\mathbb{R} \to \tildeF^u_f(0)$ satisfying the followings
			\begin{enumerate}
				\item $h_f \circ R_1=T_f^{e_d} \circ h_f$;
				\item $T_f\circ R_1=R_1\circ T_f$, where $T_f \triangleq h_f^{-1} \circ T_f^{e_1} \circ h_f :\RR \to \RR$.
			\end{enumerate}
			Consequently, $T_f^{e_1}$ induces a $C^k$ diffeomorphism $\overline{T_f}$ on $S^1$ such that
			$\pi \circ T_f = \overline{T_f} \circ \pi$.
		\end{claim}
		
		\begin{proof}[Proof of Claim \ref{claim f circle map}]

		We would like to define the conjugacy $h_f$ locally and extend it to $ \mathbb{R} $ by $ T^{e_d}_f $. More specifically, let $\gamma:(-\e,\e) \to \tildeF^u_f(0)$ be a $C^r$ diffeomorphism onto the image and $\e$ be small enough such that
		$T_f^{e_d}\big(\gamma(-\e,\e) \big)\cap \gamma(-\e,\e) =\emptyset $. This can be done by the $C^r$ leaf $\tildeF^u_f(0)$ and the $C^k$ diffeomorphism $T_f^{e_d}|_{\tildeF^u_f(0)}$ with $T_f^{e_d}(0)\ne 0$.
	
		Then, we could define a $C^k$ diffeomorphism onto the image,
		$h_0:(-\epsilon,1] \to \tildeF^u_f(0)$ such that
		\begin{equation}         		
			h_0(x)= \begin{cases}
				\gamma(x),    & x\in (-\e,\e);\\
				T_f^{e_d}\circ \gamma(x-1),      &x\in (1-\epsilon,1];\\
				\varphi(x),        & x\in [\e,1-\e].
			\end{cases}
		\end{equation}     	
		where $\varphi$ is a $C^k$ diffeomorphism onto the image and can be chosen arbitrarily. 
		Let 
		\begin{align*}
			h_f:\mathbb{R} &\to \tildeF^u_f(0),\\
			h_f(x) \triangleq (T_f^{e_d})^{[x]} &\circ h_0(x-[x]), \quad \forall x\in \mathbb{R}. 
		\end{align*}
		where $[x]$ stands for the integer part of $x$. By the construction, one can verify that $h_f$ and $T_f$ are both $C^k$ diffeomorphisms and $h_f \circ R_1=T_f^{e_d} \circ h_f$ directly. And $T_f\circ R_1=R_1\circ T_f$ is guaranteed by the commutativity of $T_f^n$, see Proposition\ref{prop}. 
		Indeed,
		\begin{align*}
			T_f\circ R_1 &=h_f^{-1} \circ T_f^{e_1} \circ h_f \circ R_1 
			=h_f^{-1} \circ T_f^{e_1} \circ T_f^{e_d}\circ h_f\\
			&=h_f^{-1} \circ T_f^{e_d} \circ T_f^{e_1}\circ h_f
			=(h_f^{-1} \circ T_f^{e_d}\circ h_f) \circ (h_f^{-1} \circ T_f^{e_1}\circ h_f)=R_1\circ T_f .
		\end{align*}      
	Then we obtain the desired diffeomorphism $h_f$ and hence a $C^k$ diffeomorphism $\overline{T_f}:S^1\to S^1$.
	\end{proof}
	
	\begin{claim}\label{claim A circle map}
		There exists a $C^{\infty}$ diffeomorphism $h_A:\RR \to \tildeF^u_A(0)$ satisfying the followings
		\begin{enumerate}
			\item $h_A \circ R_1=T_A^{e_d} \circ h_A$;
			\item  $R_\alpha= h^{-1}_A\circ T^{e_1}_A\circ h_A$, where
			 $\alpha$ is an irrational algebraic number. 
		\end{enumerate}
		In particular, $R_\alpha$ induces a rotation $\overline{R_\alpha}$ on $S^1$. Moreover if $d=2$ (the $\TT^2$ case),  one has ${\rm deg}(\alpha)=2$. 
	\end{claim}

\begin{proof}[Proof of Claim \ref{claim A circle map}]
	Let $h_A:\mathbb{R}\to \tildeF^u_A(0)$ be the linear map such that
	$h_A(0)=0\in \mathbb{R}^d$ and $h_A(1)=T_A^{e_d}(0)$.
	Then, $h_A^{-1} \circ T_A^{e_1} \circ h_A$ is actually a translation
	$R_{\alpha}(x)=x+\alpha, x\in \mathbb{R}$.
	By elementary calculate, $\alpha=x_1/x_d$ where $\vec{v}=(x_1,...,x_d)$ (given under the an orthonormal basis $\{ e_i\}_{i=1}^d$) is an eigenvector of $A$ in
	$\tildeF^s_A(0)$. Then $\alpha$ is an irrational algebraic number. Indeed, the irrational eigenvectors of $A$ implie that there is at least a pair of irrationally related coordinates $(x_i,x_j),(i\ne j)$ of $\vec{v}$ which we may assume that is $(x_1,x_d)$ and by the fact that the set of algebraic numbers is a field. Moreover, $\alpha=x_1/x_d$ is a quadratic irrational in the case of $d=2$.
\end{proof}
		
		Since $H\circ T_f^{e_1}=T_A^{e_1}\circ H$ (see Proposition \ref{prop}),   $H$ also induces a conjugacy $\overline{H}:S^1\to S^1$ from $\overline{T_f}$ to $\overline{R_{\alpha}}$.  Indeed, let $\widehat{H}\triangleq h_A^{-1} \circ H\circ h_f:\RR\to \RR$. Then  by Proposition \ref{prop}, Claim \ref{claim f circle map} and Claim \ref{claim A circle map}, one  has
		\begin{enumerate}
			\item $\widehat{H}\circ R_1=R_1\circ \widehat{H}$;
			\item $R_\alpha \circ \widehat{H} =\widehat{H}\circ T_f$.
		\end{enumerate}
	In particular,  $\widehat{H}:\RR\to \RR$ induces $\overline{H}:S^1\to S^1$ with  $\pi\circ \widehat{H}=\overline{H} \circ \pi$. 
		Moreover, $\overline{H} \circ \overline{T_f} =\overline{R_{\alpha}}\circ \overline{H}$. Namely, we have the following commutative diagram:
		
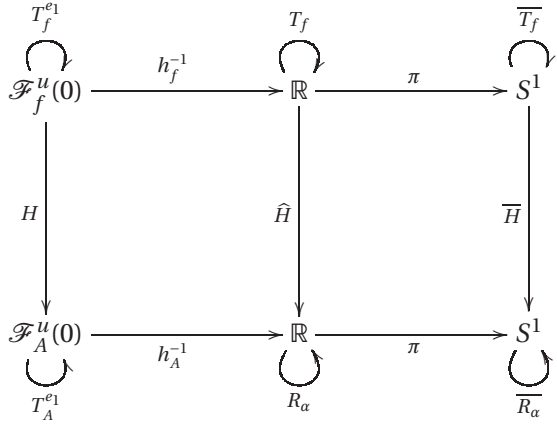
\[ \xymatrix@=15ex{
			\mathcal{F}_f^u(0)\ar[r]^{h_f^{-1}} \ar[d]_{H} \ar@(lu,ru)^{T_f^{e_1}} 
			&  \mathbb{R}\ar[r]^{\pi} \ar[d]_{\widehat{H}} \ar@(lu,ru)^{T_f}  
			& S^1\ar[d]_{\overline{H}} \ar@(lu,ru)^{\overline{T_f}} 
			\\ \mathcal{F}_A^u(0)\ar[r]_{h_A^{-1}} \ar@(ld,rd)_{T_A^{e_1}}              
			&  \mathbb{R}\ar[r]_{\pi} \ar@(ld,rd)_{R_\alpha}               
			& S^1\ar@(ld,rd)_{\overline{R_\alpha}} 
		} 
		\]
		
		By Theorem\ref{thm5}, $\overline{H}^{\pm 1}$ is absolutely continuous, so is $\widehat{H}^{\pm 1}:\RR\to \RR$.  Note that in the case of $\TT^2$ and $k=1+{\rm AC}$, one has that both $T_f, T_f^{-1}$ are $C^{1+{\rm AC}}$-smooth. It follows that there is  $C>1$ such that  $|T^{''}_f(x)|<C$ and $|T^{'}_f(x)|>\frac{1}{C}$   for Lebesgue-almost everywhere $x\in S^1$.  Hence  $T^{''}_f/T_f'\in L^2$.

		Thus $H=h_A\circ \widehat{H} \circ h_f^{-1}$ and $H^{-1}$ are also absolutely continuous along unstable leaves. Finally, by Lemma \ref{lem AC to Cr}, $H$ is $C^r$-smooth restricted on unstable leaves.
	\end{proof}

	\section*{Acknowledgements}
	\addcontentsline{toc}{section}{Acknowledgements}
	\qquad We are grateful for the valuable communication and suggestions from Yi Shi.

	\bibliographystyle{plain}
	
	\bibliography{Smooth_Stable_Foliations_of_Anosov_diffeomorphisms}

		\flushleft{\bf Ruihao Gu} \\
		Shanghai Center for Mathematical Sciences, Fudan University, Shanghai 200433, Peoples Republic of China\\
	\textit{E-mail:} \texttt{ruihaogu@fudan.edu.cn}\\

\end{document}